\documentclass[12pt]{amsart}
\usepackage[a4paper, total={5.35 in, 8.3 in}]{geometry}
\usepackage{amsmath,amsfonts,amssymb}
\usepackage{amsfonts}
\usepackage{verbatim}
\usepackage{enumerate}
\usepackage{amsthm}
\usepackage{url}
\usepackage[colorlinks]{hyperref}
\usepackage[english]{babel}
\usepackage{tikz}
\newcommand{\I}{\mathcal I}
\newcommand{\C}{\mathbb C}
\newcommand{\Tr}{\mathrm{Tr}}
\newcommand{\F}{\mathbb{F}}

\newcommand{\fqn}{\mathbb{F}_{q^n}}
\newcommand{\ord}{\mathrm{ord}}
\newcommand{\circc}{\circ^q}
\newtheorem{theorem}{Theorem}[section]
\newtheorem{conjecture}{Conjecture}
\newtheorem{prop}[theorem]{Proposition}
\newtheorem{define}[theorem]{Definition}

\newtheorem{lemma}[theorem]{Lemma}
\newtheorem{cor}[theorem]{Corollary}
\newtheorem{example}[theorem]{Example}
\newtheorem{remark}[theorem]{Remark}
\author[Lucas Reis]{Lucas Reis}
\address{
Departamento de Matem\'{a}tica\\
Universidade Federal de Minas Gerais\\
UFMG\\
Belo Horizonte, MG\\
 30123-970\\
 Brazil\\
 }
\email{lucasreismat@gmail.com}

\title{Existence results on k-normal elements over finite fields}

\date{\today
}


\begin{document}
\maketitle

\begin{abstract}
An element $\alpha \in \F_{q^n}$ is normal over $\F_q$ if $\alpha$ and its conjugates $\alpha, \alpha^q, \cdots \alpha^{q^{n-1}}$ form a basis of $\F_{q^n}$ over $\F_q$. Recently, Huczynska, Mullen, Panario and Thomson (2013) introduce the concept of $k$-normal elements, generalizing the normal elements. In the past few years, many questions concerning the existence and number of $k$-normal elements with specified properties have been proposed. In this paper, we discuss some of these questions and, in particular, we provide many general results on the existence of $k$-normal elements with additional properties like being primitive or having large multiplicative order. We also discuss the existence and construction of $k$-normal elements in finite fields, providing a connection between $k$-normal elements and the factorization of $x^n-1$ over $\F_q$.
\end{abstract}

\section{Introduction}
Let $\F_{q^n}$ be the finite field with $q^n$ elements, where $q$ is a prime power and $n$ is a positive integer. The field $\F_{q^n}$ has two main algebraic structures: the set of nonzero elements $\F_{q^n}^*$ is a cyclic group and $\fqn$ is an $n$-dimensional vector space over $\F_q$. We have a notion of generators in these algebraic structures: an element $\alpha\in \fqn^*$ is primitive if $\alpha$ generates the cyclic group $\fqn^*$ or, equivalently, $\alpha$ has multiplicative order $q^n-1$. Also, $\alpha\in \fqn$ is normal over $\F_q$ if $\alpha$ and its conjugates $\alpha^{q}, \ldots, \alpha^{q^{n-1}}$ span $\fqn$ as an $\F_q$-vector space: in this case, the set $\{\alpha, \alpha^{q}, \ldots, \alpha^{q^{n-1}}\}$ is a normal basis. Primitive elements are constantly used in cryptographic applications such as discrete logarithm problem and pseudorandom number generators. Normal bases are also object of interest in many applications due to their efficiency in fast arithmetic: for instance, the $q$-th power map $a\mapsto a^q$ is performed by the shift operator in a normal basis. Sometimes it is interesting to combine these properties and therefore obtain elements that are simultaneously primitive and normal. The celebrated Primitive Normal Basis Theorem says that, for any positive integer $n$ and any finite field $\F_q$, there exists an element $\alpha\in \fqn$ that is primitive and normal over $\F_q$: this result was first proved by Lenstra and Schoof \cite{LS87} and a proof without the use of computers was latter given by Cohen and Huczynska \cite{CH03}.

Recently, Huczynska et al. \cite{HMPT} introduce the concept of $k$-normal elements, as an extension of the usual normal elements. In \cite{HMPT}, they show many equivalent definitions for these $k$-normal elements and here we pick the most natural in the sense of vector spaces.

\begin{define}
Given $\alpha\in \F_{q^n}$, let $V_{\alpha}$ be the $\F_q$-vector space generated by $\alpha$ and its conjugates $\alpha^q, \ldots, \alpha^{q^{n-1}}$. The element $\alpha$ is $k$-normal over $\F_q$ if $V_{\alpha}$ has dimension $n-k$, i.e., $V_{\alpha}$ has co-dimension $k$.
\end{define}
From this definition, the usual normal elements correspond to the $0$-normal elements and the element $0\in \fqn$ is the only $n$-normal element in $\fqn$. Of course, the number $k$ is always in the interval $[0, n]$. The concept of $k$-normality depends strongly on the base field that we are working: for instance, if $\F_8=\F_2(\beta)$ with where $\beta^3=\beta+1$, then $\beta$ is normal over $\F_8$ but is $1$-normal over $\F_2$: its conjugates $\beta^2$ and $\beta^4$ over $\F_2$ satisfy $\beta^4=\beta^2+\beta$. Unless otherwise stated, $\alpha\in \fqn$ is $k$-normal if it is $k$-normal over $\F_q$.

Motivated by the Primitive Normal Basis Theorem, in 2014, Mullen and Anderson propose the following problem (see \cite{M16}, Conjecture 3).
\begin{conjecture}\label{conj:mullen}
Suppose that $p\ge 5$ is a prime and $n\ge 3$. Then, for $k=0, 1$ and $a=1, 2$, there exists some $k$-normal element $\alpha\in \F_{p^n}$ with multiplicative order $\ord(\alpha)=\frac{p^n-1}{a}$.
\end{conjecture}

We observe that the case $(a, k)=(1, 0)$ concerns the existence of normal elements that are also primitive and is covered by the Primitive Normal Basis Theorem (which is not restricted to prime fields). Recently~\cite{RT18}, the authors prove the case $(a, k)=(1, 1)$ and, in~\cite{KR18}, the authors propose a solution for the cases $(a, k)=(2, 0)$ and $(2, 1)$: in both cases, the proposed conjecture is not only established, but is also extended to general finite fields (removing the restriction to prime fields). 

The previous problem can be viewed as small variations of the Primitive Normal Basis Theorem: elements with multiplicative order $(q^n-1)/2$ and $1$-normal elements are included. In this paper, we discuss some questions concerning variations of this problem and, in particular, we provide partial answers to some general problems that were previously proposed in~\cite{HMPT}. Here we make a short overview of our main results. 

We obtain sufficient conditions for the existence of $k$-normal elements that are also primitive. Considering a relaxation of the primitivity condition, we obtain existence results on $k$-normal elements with large multiplicative order. In general, all of these results are stated under the natural condition that $k$-normal elements actually exist. For the sake of completeness, we further explore the extensions $\F_{q^n}$ of $\F_q$ where we actually have $k$-normal elements for any $0\le k\le n$: in these cases, our existence results are, somehow, unconditional. We further provide special classes of positive integers $n$ for which this phenomena occurs: in this case, the number $n$ yields an ``enumerator polynomial'' whose coefficients provide the exact number of $k$-normal elements in $\F_{q^n}$, for each integer $k$ in the interval $[0, n]$.

\section{Preliminaries}
In this section, we provide a background material that is used along the paper. We start with some arithmetic functions and their polynomial versions. Throughout this paper, for a positive integer $n$, $\varphi(n)$ denotes the Euler Phi function, $\mu(n)$ is the M\"{o}bius function and $d(n)$ denotes the number of distinct divisors of $n$.

\begin{define}\label{def:functions}
\begin{enumerate}[(a)]
\item Let $f(x)$ be a monic polynomial with coefficients in $\F_q$. The Euler Phi Function for polynomials over $\F_q$ is given by $$\Phi_q(f)=\left |\left(\frac{\F_q[x]}{\langle f\rangle}\right)^{*}\right |,$$ where $\langle f\rangle$ is the ideal generated by $f(x)$ in $\F_q[x]$. 
\item If $t$ is a positive integer (or a monic polynomial over $\F_q$), $W(t)$ denotes the number of squarefree (monic) divisors of $t$.
\item If $f(x)$ is a monic polynomial with coefficients in $\F_q$, the Polynomial M\"{o}bius Function $\mu_q$ is given by $\mu_q(f)=0$ if $f$ is not squarefree and $\mu_q(f)=(-1)^r$ if $f$ is a product of $r$ distinct irreducible factors over $\F_q$.  
\end{enumerate}
\end{define}

\subsection{Estimates}\label{subsec:estimates}
Here we present some estimates that are used along the next sections. First, from the main result in \cite{NR83}, we easily obtain the following lemma.
\begin{lemma}\label{lem:estimate-divisor} If $d(m)$ is the number of divisors of $m$, then for all $m\ge 3$,
$$d(m)<m^{\frac{1.1}{\log\log m}}.$$
\end{lemma}

From the previous lemma, the following result is straightforward.

\begin{lemma}\label{lem:estimate}
For any $m\ge 3$ and $x>0$,
\begin{equation*}\sum_{d|m\atop{d\le x}}\varphi(d)< x\cdot m^{\frac{1.1}{\log\log m}}.\end{equation*}
\end{lemma}

For the number of squarefree divisors function, we have the following bounds.
\begin{lemma}\label{lem:estimate-divisor2}
If $W(t)$ is the number of squarefree divisors of $t$, then for all $t\ge 3$,
$$W(t-1)<t^{\frac{0.96}{\log\log t}}.$$
\end{lemma}
\begin{proof} Let $w(t)$ be the number of distinct prime divisors of $t$. According to Ineq.(4.1) of~\cite{COT}, $w(t-1)\le \frac{1.38402\cdot \log t}{\log \log t}$ if $t\ge 3$. Therefore, $$W(t-1)=2^{w(t-1)}\le 2^{\frac{1.38402\log t}{\log \log t}}=t^{\log 2\cdot \frac{1.38402}{\log \log t}}<t^{\frac{0.96}{\log\log t}}.$$\end{proof}

\begin{lemma}[see Lemma~3.3 of~\cite{CH03}]\label{lem:cohen-hucz}
For any integer $m$,
$$W(m)<4.9\cdot m^{1/4}.$$
\end{lemma}

\begin{prop}\label{prop:estimate-euler}
For any positive integers $k$ and $q\ge 2$, the following holds:
$$(q-1)^k\ge 4^{-k/q}q^k.$$
\end{prop}
\begin{proof} We observe that $a_n=(1-1/n)^n$ is an increasing sequence. Therefore, for any $q\ge 2$, $a_q\ge a_2=4^{-1}$ and so $$(q-1)^k=((q-1)^q)^{k/q}=(a_q\cdot q^{q})^{k/q}\ge 4^{-k/q}q^k.$$\end{proof}

\subsection{Linearized polynomials and the $\F_q$-order}
Here we present some definitions and basic results on linearized polynomials over finite fields that are frequently used in this paper.

\begin{define}Let $f\in \F_q[x]$ with $f(x)=\sum_{i=0}^ra_ix^i$.
\begin{enumerate}[(a)]
\item The polynomial $L_f(x):=\sum_{i=0}^ra_ix^{q^{i}}$ is the $q$-associate of $f$.
\item For $\alpha\in \fqn$, we set $f\circc \alpha=L_f(\alpha)=\sum_{i=0}^ra_i\alpha^{q^{i}}$.
\end{enumerate}
\end{define}

The $q$-associates have interesting additional properties.

\begin{theorem}\label{thm:linearized}
Let $f, g\in \F_q[x]$. The following hold:
\begin{enumerate}[(a)]
\item $L_f(x)+L_g(x)=L_{f+g}(x)$,
\item $L_{fg}(x)=L_f(L_g(x))=L_g(L_f(x))$.
\end{enumerate}
\end{theorem}
\begin{proof}
This result follows by direct calculations. For more details, see Section 3.4 of \cite{LN}.
\end{proof}

For $\alpha\in\fqn$, we set $\I_{\alpha}=\{f\in\F_q[x]\;|\; f\circc \alpha=0\}$. From Theorem~\ref{thm:linearized}, $\I_{\alpha}$ is an ideal and is direct to verify that $(x^n-1)\circc \alpha=\alpha^{q^n}-\alpha=0$. Therefore, $\I_{\alpha}$ contains $x^n-1$ and so $\I_{\alpha}\ne \{0\}$ is generated by a nonzero polynomial $m_{\alpha, q}(x)$, which we can suppose to be monic. 
\begin{define}
The polynomial $m_{\alpha, q}\in \F_q[x]$ is defined as the $\F_q$-order of $\alpha$.
\end{define}
For instance, the $\F_q$-order of  the element $0$ is $m_{0, q}(x)=1$ and, for any $\alpha\in \F_q^*$, $m_{\alpha, q}(x)=x-1$ (since $\alpha^q-\alpha=0$). In general, for $\alpha\in \fqn$, $m_{\alpha, q}(x)$ divides $x^n-1$ and so this polynomial has degree $k$ for some $0\le k\le n$. In \cite{HMPT}, the authors show that an element $\alpha$ is $k$-normal if and only if its $\F_q$-order is a polynomial of degree $n-k$ (see Theorem 3.2 of~\cite{HMPT}). Their proof uses the rank of a specific matrix that measures the dimension of the $\F_q$-vector space $V_{\alpha}$ generated by $\alpha$ and its conjugates. Here we present a much simpler and shorter proof of this result.

\begin{prop}\label{prop:char-k-normal}
Let $\alpha \in \F_{q^n}$ be an element with $\F_q$-order $h$ and let $V_{\alpha}$ be the $\F_q$-vector space generated by $\alpha$ and its conjugates $\alpha^q, \ldots, \alpha^{q^{n-1}}$. Then $V_{\alpha}$ has dimension $\deg(h)$. In particular, $\alpha$ is $k$-normal if and only if $m_{\alpha, q}(x)$ has degree $n-k$. 
\end{prop}
\begin{proof} Let $s=\deg(h)$ and let $V_n$ be the set of polynomials of degree at most $n-1$ over $\F_q$. Of course, $V_n$ is an $\F_q$-vector space of dimension $n$: the elements $1, x, \ldots, x^{n-1}$ form a basis for $V_n$. Let and let $\tau_{\alpha}: V_n\to V_{\alpha}$ be defined as follows: for $f\in V_n$ with $f=\sum_{i=0}^{n-1}a_ix^i$, set $\tau_{\alpha}(f)=f\circc \alpha=\sum_{i=0}^{n-1}a_i\alpha^{q^i}$. From definition, the map $\tau_{\alpha}$ is linear and onto. From the definition of $\F_q$-order, the kernel of $\tau_{\alpha}$ is given by $$U_{\alpha}=\{g\cdot h\,|\, g\in \F_q[x]\;\,\text{and}\;\,\deg(g)\le n-1-s\},$$ which is an $\F_q$-vector space isomorphic to $V_{n-s}$. The rank-nullity theorem yields $\dim V_{\alpha}=n-\dim U_{\alpha}=n-(n-s)=s$.\end{proof}

In particular, an element $\alpha$ is normal if and only if $m_{\alpha, q}(x)=x^n-1$. As follows, we have a formula for the number of $k$-normal elements.

\begin{lemma}[see Theorem 3.5 of~\cite{HMPT}]\label{eq:count} The number $N_k$ of $k$-normal elements of $\F_{q^n}$ over $\F_q$ is given by
\begin{equation*}\label{eq:num-k-normal}N_k=\sum_{h|x^n-1\atop{\deg(h)=n-k}}\Phi_q(h),\end{equation*}
where the divisors are monic and polynomial division is over $\F_q$.
\end{lemma}

The proof of the previous lemma follows from Proposition~\ref{prop:char-k-normal} and the fact that, for each monic divisor $f$ of $x^n-1$, there exist $\Phi_q(f)$ elements $\alpha\in \F_{q^n}$ for which $m_{\alpha, q}(x)=f$ (see Theorem 11 of~\cite{O34}). 

\begin{remark}\label{remark:k-normals-div}We observe that there exist $k$-normal elements in $\fqn$ if and only if $x^n-1$ is divisible by a polynomial $h\in \F_q[x]$ of degree $n-k$ (or, equivalently, $x^n-1$ is divisible by a polynomial $h\in \F_q[x]$of degree $k$). This equivalence is frequently used along the paper. 
\end{remark}

\subsection{Characters and characteristic functions}
Here we use the additive-multiplicative character method of  Lenstra and Schoof in the construction of certain characteristic functions over finite fields. This method has been used by many different authors in a wide variety of existence problems. For this reason, we skip some details, which can be found in \cite{CH03}.  We introduce the notion of {\it freeness}.

\begin{define}
\begin{enumerate}
\item If $m$ divides $q^n-1$, an element $\alpha \in \F_{q^n}^*$ is $m$-free if $\alpha = \beta^d$ for any divisor $d$ of $m$ implies $d=1$ and $\alpha=\beta$. 
\item If $m(x)$ divides $x^n-1$, an element $\alpha\in \F_{q^n}$ is $m(x)$-free if $\alpha = h \circc \beta=L_h(\beta)$ for any divisor $h(x)$ of $m(x)$ implies $h=1$ and $\alpha=\beta$. 
\end{enumerate}
\end{define}

\begin{remark}\label{remark:freeness}
It is well-known that primitive elements correspond to the $(q^n-1)$-free elements and normal elements correspond to the $(x^n-1)$-free elements (see Proposition 5.2 and Theorem 5.3 of~\cite{HMPT}).
\end{remark}
\subsubsection{Characters and freeness}\label{subsub:char-free}
Let $\alpha$ be a primitive element of $\fqn$. A typical multiplicative character $\eta$ of $\fqn^*$ is a function $\eta:\fqn^*\to \C$ given by $\eta(\alpha^k)=e^{\frac{2\pi i dk}{q^n-1}}$ for some positive integer $d$. Its order is the least positive integer $s$ such that $\eta(\beta)^s=1$ for any $\beta \in \fqn^*$. The set of distinct multiplicative characters of $\fqn^*$ is a cyclic group, isomorphic to $\fqn^*$. The character $\eta_1$ given by $\eta_1(\alpha^k)=1$ is the trivial multiplicative character. As usual, we extend the multiplicative characters to $0\in \F_{q^n}$ by setting $\eta(0)=1$ if $\eta$ is trivial and $\eta(0)=0$, otherwise.

If $p$ is the characteristic of $\F_q$ and $q=p^s$, let $\chi$ be the mapping $\chi:\fqn\to \C$ defined by
$$\chi(\alpha)=e^{\frac{2\pi \Tr_{q^n/p}(\alpha)}{p}}, \alpha\in \fqn,$$
where $\Tr_{q^n/p}(\alpha)=\sum_{i=0}^{ns-1}\alpha^{p^i}\in \F_p$ is the trace function of $\fqn$ over $\F_p$. In this case, $\chi$ is the canonical additive character. Moreover, for each $\beta \in \fqn$, the mapping $\chi_{\beta}$ given by $\chi_{\beta}(\alpha)=\chi(\beta\alpha)$ for $\alpha\in \fqn$ is another additive character of $\fqn$ and, in fact, any additive character of $\fqn$ is of this form. The set of additive characters is a group $G$, isomorphic to $\F_{q^n}$ (as an additive group) via the map $\tau:\F_{q^n}\to G$ given by $\tau(\beta)=\chi_{\beta}$. In this correspondence, $\chi_{\beta}$ has $\F_q$-order $h$ if $\beta$ has $\F_q$-order $h$. Of course, $\chi_{0}(\alpha)=1$ for every $\alpha\in \F_q$; $\chi_0$ is the trivial additive character.

The concept of freeness derives some characteristic functions for primitive and normal elements. We pick the notation of \cite{HMPT}.

\vspace{.2cm}
{\bf Multiplicative Part}: $\int\limits_{d|m}\eta_d$ stands for the sum $\sum_{d|t}\frac{\mu(d)}{\varphi(d)}\sum_{(d)}\eta_d$, where $\eta_d$ is a typical multiplicative character of $\F_{q^n}$ of order $d$, and the sum $\sum_{(d)}\eta_d$ runs over all the multiplicative characters of order $d$. It is worthy of mentioning that there exist $\varphi(d)$ multiplicative characters of order $d$.

\vspace{.2 cm}
{\bf Additive Part}:  If $D$ is a monic divisor of $x^n-1$ over $\F_q$, let $\Delta_D$ be the set of all $\delta\in \F_{q^n}$ such that $\chi_{\delta}$ has $\F_q$-order $D$. For instance, $\Delta_1=\{0\}$ and $\Delta_{x-1}=\F_q^*$. Analogously, $\int\limits_{D|T}\chi_{\delta_D}$ stands for the sum $\sum_{D|T}\frac{\mu_q(D)}{\Phi_q(D)}\sum_{(\delta_D)}\chi_{\delta_D}$, where $\chi_{\delta_D}$ is a typical additive character of $\F_{q^n}$ of $\F_q$-Order $D$ and the sum $\sum_{(\delta_D)}\chi_{\delta_D}$ runs over all the additive characters whose $\F_q$-order equals $D$, i.e., $\delta_D\in \Delta_D$. For instance, $\sum_{(\delta_1)}\chi_{\delta_1}=\chi_0$ and $\sum_{(\delta_{x-1})}\chi_{\delta_{x-1}}=\sum_{c\in \F_q^*}\chi_c$. In general, there are $\Phi_q(D)$ multiplicative additive characters of $\F_q$-order $D$.

For each divisor $t$ of $q^n-1$ and each monic divisor $T\in \F_q[x]$ of $x^n-1$, set $\theta(t)=\frac{\varphi(t)}{t}$ and $\Theta(T)=\frac{\Phi_q(T)}{q^{\deg T}}$. 
\begin{theorem}[see Section 5.2 of \cite{HMPT}]\label{thm:charfree}The following hold.
\begin{enumerate}
\item For $w \in \fqn^*$ and $t$ be a positive divisor of $q^n-1$, 
\[\omega_t(w) = \theta(t) \int_{d|t} \eta_{d}(w) = \begin{cases} 1 & \text{if $w$ is $t$-free,} \\ 0 & \text{otherwise.} \end{cases}\]
\item For $w \in \fqn$ and $D$ be a monic divisor of $x^n-1$,
\[\Omega_T(w) = \Theta(T) \int_{D|T} \chi_{\delta_D}(w) = \begin{cases} 1 & \text{if $w$ is $D$-free,} \\ 0 & \text{otherwise.} \end{cases}\]
\end{enumerate}
\end{theorem}
In particular, for $t=q^n-1$ and $T=x^n-1$, $\omega_t$ and $\Omega_T$ are the characteristic functions for primitive and normal elements, respectively (see Remark~\ref{remark:freeness}). 

\subsubsection{Character sums estimates}\label{subsubsec:char}
Here we present estimates for certain character sums that are of our interest. We first observe that, if $\chi$ is a (nontrivial) additive character of $\fqn$, $\sum_{c\in \fqn}\chi(c)=0$. Additionally, if $\chi$ and $\eta$ are the trivial additive and multiplicative characters of $\fqn$, respectively, then $\sum_{c\in \fqn}\chi(c)\eta(c)=q^n$. We further require estimates of more complicated character sums. The following results are useful.

\begin{lemma}[see Theorem 5.41 of~\cite{LN}]\label{Gauss1} Let $\eta$ be a multiplicative character of $\F_{q^s}$ of order $r>1$ and $F\in \F_{q^s}[x]$ be a monic polynomial of positive degree such that $F$ is not of the form $g(x)^r$ for some $g\in \F_{q^s}[x]$ with degree at least $1$. Suppose that $e$ is the number of distinct roots of $F$ in its splitting field over $\F_{q^s}$. For every $a\in \F_{q^s}$, 
$$\left|\sum_{c\in \F_{q^s}}\eta(aF(c))\right|\le (e-1)q^{s/2}.$$
\end{lemma}


\begin{lemma}[see Theorem~2G of~\cite{schmidt}]\label{Gauss2}
Let $\eta$ be a multiplicative character of $\F_{q^s}$ of order $d\neq 1$ and $\chi$ a non-trivial additive character of $\F_{q^s}$. If $F, G\in\F_{q^s}[x]$ are such that $F$ has exactly $m$ roots and $\deg(G)=n$ with $\gcd(d,\deg(F))=\gcd(n,q)=1$, then
\[
\left| \sum_{c\in\F_{q^s}} \eta(F(c))\chi(G(c)) \right| \leq (m+n-1) q^{s/2} .
\]
\end{lemma}

\section{On primitive $k$-normal elements}
As follows, we may construct $k$-normal elements from a given normal element.
\begin{lemma}\label{LR1}
Let $\beta\in \F_{q^n}$ be a normal element over $\F_q$ and $f\in \F_q[x]$ be a polynomial of degree $k$ such that $f$ divides $x^n-1$. Then $\alpha=f\circc \beta$ is $k$-normal.
\end{lemma}
\begin{proof} We prove that $m_{\alpha, q}(x)=\frac{x^n-1}{f}$ and, from Proposition~\ref{prop:char-k-normal}, this implies the desired result. We observe that $$\frac{x^n-1}{f}\circc \alpha=\frac{x^n-1}{f}\circc (f\circc \beta)=(x^n-1)\circc \beta=\beta^{q^n}-\beta=0,$$ hence $m_{\alpha, q}(x)$ divides $\frac{x^n-1}{f}$. If it divides strictly, there exists a monic polynomial $g\in \F_q[x]$ of degree at least one such that $m_{\alpha, q}(x)$ divides $\frac{x^n-1}{fg}$ and so
$$0=\frac{x^n-1}{fg}\circc \alpha=\frac{x^n-1}{fg}\circc (f\circc \beta)=\frac{x^n-1}{g}\circc \beta,$$
a contradiction with $m_{\beta, q}(x)=x^n-1$ and this completes the proof. \end{proof}

From Eq.~\eqref{eq:count}, there exist $k$-normal elements in $\fqn$ if and only if $x^n-1$ has a divisor of degree $n-k$ (or, equivalently, a divisor of degree $k$). In particular, we have a method for constructing $k$-normal elements when they actually exist: if we find a divisor $f\in \F_q[x]$ of $x^n-1$ of degree $k$ and a normal element $\beta\in \F_{q^n}$, the element $\alpha=f\circc \beta=L_f(\beta)\in \F_{q^n}$ is $k$-normal. There are many ways of finding normal elements in finite field extensions,  including constructive and random methods; this is a classical topic in the theory of finite fields and the reader can easily find a wide variety of papers regarding those methods. For instance, see \cite{GG90}.

With the previous observations, we obtain the characteristic function for a special class of $k$-normal elements. If we write $\omega_{q^n-1}=\omega$ and $\Omega_{x^n-1}=\Omega$, from Theorem~\ref{thm:charfree} and Remark~\ref{remark:freeness}, for $w\in \fqn$, $\Omega(w)\cdot \omega(L_f(w))=1$ if and only if $w$ is normal and $L_f(w)=f\circc w$ is primitive. This characteristic function describes a particular class of primitive $k$-normal elements: the number $n_f$ of normal elements $\alpha\in \F_{q^n}$ such that $f\circc \alpha=L_f(\alpha)$ is primitive equals $$\sum_{w\in \fqn}\Omega(w)\cdot \omega(L_f(w)).$$ 

Throughout this section, unless otherwise stated, $f\in \F_q[x]$ is a divisor of $x^n-1$ and has degree $k$. The following result is straightforward.

\begin{prop}\label{charfun}
Let $f\in \F_q[x]$ be a divisor of $x^n-1$ of degree $k$ and let $n_f$ be the number of normal elements $\alpha\in \F_{q^n}$ such that $f\circc \alpha=L_f(\alpha)$ is primitive. The following holds:
\begin{equation}\label{Char}\frac{n_f}{\theta(q^n-1)\Theta(x^n-1)}=\sum\limits_{w\in \F_{q^n}}\displaystyle\int\limits_{d|q^n-1}\displaystyle\int\limits_{D|x^n-1}\eta_d(L_{f}(w))\chi_{\delta_D}(w).\end{equation}
In particular, if $n_f>0$, there exist primitive, $k$-normal elements in $\fqn$.
\end{prop}

\subsection{A sufficient condition for the existence of primitive $k$-normals}
We make some estimates for the character sums that appear in Eq.~\eqref{Char}. We observe that $\eta_d$ is the trivial multiplicative character if and only if $d=1$. Also, $\chi_{\delta_D}$ is the trivial additive character if and only if $\delta_D=0$, i.e., $D=1$. As usual, we split the sum in Eq.~\eqref{Char} as Gauss sums types, according to the trivial and non-trivial characters. For each divisor $d$ of $q^n-1$ and each divisor $D\in \F_q[x]$ of $x^n-1$, set $G_f(\eta_d, \chi_{\delta_D})=\sum\limits_{w\in \F_{q^n}}\eta_d(L_{f}(w))\chi_{\delta_D}(w)$.

Observe that, from Proposition~\ref{charfun}, $$\frac{n_f}{\theta(q^n-1)\Theta(x^n-1)}=S_0+S_1+S_2+S_3,$$ where $S_0=G_f(\eta_1, \chi_0)$, $S_1=\displaystyle\int\limits_{D|x^n-1\atop D\ne 1}G_f(\eta_1, \chi_{\delta_D})$, $S_2=\displaystyle\int\limits_{d|q^n-1\atop d\ne 1}G_f(\eta_d, \chi_{0})$ and $$S_3=\displaystyle\int\limits_{d|q^n-1\atop d\ne 1}\displaystyle\int\limits_{D|x^n-1\atop D\ne 1}G_f(\eta_d, \chi_{\delta_D}).$$ 

We observe that $S_0=q^n$ and $$G_f(\eta_1, \chi_{\delta_D})=\sum\limits_{w\in \F_{q^n}}\eta_1(L_{f}(w))\chi_{\delta_D}(w)=\sum\limits_{w\in \F_{q^n}}\chi_{\delta_D}(w)=0,$$ for any divisor $D$ of $x^n-1$ with $D\ne 1$ (see Subsection~\ref{subsubsec:char}). In particular, $S_1=0$. We require good estimates on the sums $S_2$ and $S_3$. We observe that, in these sums, each term of the form $G_f(\eta_d, \chi_{\delta_D})$ comes with a weight $\mu(d)\mu_q(D)$ (see the compact notation of integrals in Subsection~\ref{subsub:char-free}), which is zero if $d$ or $D$ is not squarefree. In particular, only summands with both $d$ and $D$ squarefree have influence on the sums $S_2$ and $S_3$. Hence, $S_2$ has $W(q^n-1)-1$ nonzero terms and $S_3$ has $(W(q^n-1)-1)(W(x^n-1)-1)$ nonzero terms (see Definition~\ref{def:functions}). Therefore, for $$M_2=\max\limits_{d|q^n-1\atop d\ne 1}|G_f(\eta_d, \chi_{0})|\quad \text{and}\quad M_3=\max\limits_{d|q^n-1, D|x^n-1\atop d, D\ne 1}|G_f(\eta_d, \chi_{\delta_D})|,$$ 
we have $|S_2|\le M_2\cdot (W(q^n-1)-1)$ and $|S_3|\le M_3\cdot (W(q^n-1)-1)(W(x^n-1)-1)$.  We apply Lemmas~\ref{Gauss1} and~\ref{Gauss2} to obtain bounds for $M_2$ and $M_3$. Write $f(x)=\sum_{i=0}^k a_ix^i$. Since $f$ divides $x^n-1$, $f$ is not divisible by $x$ and so $a_0\ne 0$. One can see that the formal derivative of the $q$-associate $L_f(x)=\sum_{i=0}^{k}a_ix^{q^i}$ of $f$ equals $a_0$. In particular, $L_f$ does not have repeated roots, hence is not of the form $y\cdot G(x)^r$ for any $y\in \F_{q^n}, G\in \F_{q^n}[x]$ and $r>1$.  Also, if $f$ divides $x^n-1$ and has degree $k$, $L_f(x)$ has degree $q^k$ and so $\gcd(\deg(L_f), q^n-1)=1$ and the equation $L_f(x)=0$ has exactly $q^{k}$ distinct solutions over $\F_{q^n}$; these solutions describe a $k$-dimensional $\F_q$-vector subspace of $\F_{q^n}$. Finally, we observe that $G(x)=x$ cannot be written as $h^p-h-y$ for any rational function $h\in \F_{q^n}(x)$ and $y\in \F_{q^n}$. From Lemma~\ref{Gauss1}, we conclude that, for each divisor $d\ne 1$ of $q^n-1$, $$|G_f(\eta_d, \chi_0)|=\left|\sum_{w\in \F_{q^n}}\eta_d(L_f(w))\right|\le (q^k-1)q^{n/2}<q^{n/2+k}.$$ 
From Lemma~\ref{Gauss2}, we conclude that, for each divisor $D\ne 1$ of $x^n-1$ and each divisor $d\ne 1$ of $q^n-1$, 
$$|G_f(\eta_d, \chi_{\delta_D})|=\left|\sum_{w\in \F_{q^n}}\eta_d(L_f(w))\chi_{\delta_D}(w)\right|\le (q^k+1-1)q^{n/2}=q^{n/2+k}.$$
Combining all the previous bounds, we have the following inequalities:
\begin{equation}\label{eq:auxiliar}|S_2|\le q^{n/2+k}(W(q^n-1)-1)\;, \; |S_3|\le q^{n/2+k}(W(q^n-1)-1)(W(x^n-1)-1).\end{equation}

We are ready to state our main result.
\begin{theorem}\label{main}
Let $f\in \F_q[x]$ be a divisor of $x^n-1$ of degree $k$ and let $n_f$ be the number of normal elements $\alpha\in \F_{q^n}$ such that $f\circc \alpha$ is primitive. The following holds:
$$\label{charac}\frac{n_f}{\theta(q^n-1)\Theta(x^n-1)}>q^n-q^{n/2+k}W(q^n-1)W(x^n-1).$$
In particular, if 
\begin{equation}\label{sieveLR}  q^{n/2-k}\ge W(q^n-1)W(x^n-1),\end{equation}
then there exist primitive $k$-normal elements in $\F_{q^n}$.
\end{theorem}
\begin{proof} We have seen that $\label{charac}\frac{n_f}{\theta(q^n-1)\Theta(x^n-1)}=S_0+S_1+S_2+S_3=q^n+S_2+S_3$. In particular, 
$$\frac{n_f}{\theta(q^n-1)\Theta(x^n-1)}\ge q^n-|S_2|-|S_3|.$$
Applying the estimates given in Ineq.~\eqref{eq:auxiliar}, we obtain $$|S_2|+|S_3|<q^{n/2+k}W(q^n-1)W(x^n-1).$$ In particular, if Ineq.~\eqref{sieveLR} holds, $n_f>0$ and so there exist primitive, $k$-normal elements in $\fqn$.\end{proof}

\subsection{Asymptotic results}
Employing bounds on the functions $W(q^n-1)$ and $W(x^n-1)$, we obtain the following result.
\begin{prop}\label{prop:suf-condition}
Suppose that there exist $k$-normal elements in $\F_{q^n}$, where $n\ge 2$. If
\begin{equation}\label{eq:suf-condition}k\le n\cdot \left(\frac{1}{2}-\frac{0.96}{\log n+\log\log q}-\log_q2 \right),\end{equation}
at least one of these $k$-normal elements is also primitive.
\end{prop}
\begin{proof} 
We observe that $x^n-1$ has at most $n$ distinct irreducible divisors over $\F_q$, hence $W(x^n-1)\le 2^{n}=q^{n\log_q2}$. Since $n\ge 2$, we have that $q^n\ge 4$ and then, from Lemma~\ref{lem:estimate-divisor2} we obtain
$$W(q^n-1)\le q^{\frac{0.96 n}{\log n+\log\log q}}.$$
Using these bounds, one can see that, if Ineq.~\eqref{eq:suf-condition} holds, then $$q^{n/2-k}\ge W(q^n-1)W(x^n-1),$$ and the result follows from Theorem~\ref{main}.\end{proof}

\begin{remark}\label{remark:sharp}For $h(n, q)=\frac{1}{2}-\frac{0.96}{\log n+\log\log q}-\log_q2$, we have $\lim\limits_{q\to \infty}h(n, q)=1/2$ uniformly on $n$. In particular, given $\varepsilon>0$, for $q$ sufficiently large, we can ensure the existence of primitive $k$-normal elements 
for $k\in [0, (\frac{1}{2}-\varepsilon)n]$, whenever $k$-normal elements actually exist in $\fqn$.  
\end{remark}

In general, the bounds on character sums over $\F_{q^n}$ yield the factor $q^{n/2}$. In particular, this result is somehow sharp based on our character sums estimates. In fact, for $n=4$, $k=n/2=2$ is already a critical value: if $q\equiv 3\pmod 4$, we have that $x^4-1$ factors as $(x+1)(x-1)(x^2+1)$ over $\F_q$. In particular, the $\F_q$-order of an element that is $2$-normal equals $x^2-1$ or $x^2+1$. In other words, any element $\beta\in \F_{q^4}$ that is $2$-normal satisfies $\beta^{q^2}=\pm \beta$. Therefore, $\beta^{2(q^2-1)}=1$ and so the multiplicative order of $\beta$ is at most $2(q^2-1)<q^4-1=|\F_{q^4}^*|$. In conclusion, no $2$-normal element in $\F_{q^4}$ can be primitive. Far from the extreme $n/2$, we have effective results.

\begin{cor}\label{cor:430}
Let $q\ge 301$ be a power of a prime and let $n\ge 16$ be a positive integer such that the pair $(q, n)$ is in Table~\ref{tab:pairs}. For any $k\in [0, n/8]$, if there exist $k$-normal elements in $\F_{q^n}$, at least one of them is also primitive.
\end{cor}
\begin{proof} From Lemma~\ref{lem:cohen-hucz}, $W(q^n-1)<4.9 \cdot q^{n/4}$ and we have the trivial bound $W(x^n-1)\le 2^n$. In particular, if \begin{equation}\label{eq:n/4}\left(\frac{q}{256}\right)^{n/8}\ge 4.9,\end{equation}
then Ineq.~\eqref{sieveLR} holds true for any $k\in [0, n/8]$. It is direct to verify that Ineq.~\eqref{eq:n/4} holds true for the pairs $(q, n)$ in Table~\ref{tab:pairs}. Therefore, from Theorem~\ref{main}, if the pair $(q, n)$ is in Table~\ref{tab:pairs}, there exist primitive $k$-normal elements in $\F_{q^n}$ whenever $k$-normals exist and $k\in [0, n/8]$.
\end{proof}

\begin{table}[h]
\begin{center}\small
\begin{tabular}{c||c|c|c|c|c|c|c}
$q$ &  $\ge 567$ & $\ge 435$ & $\ge 381$& $\ge 352$ & $\ge 334$ & $\ge 301$ \\ \hline
$n$ & $\ge 16$ & $\ge 24$ & $\ge 32$ & $\ge 40$ & $\ge 48$ & $\ge 80$ 
\end{tabular}
\end{center}
\caption{Values for $q$ and $n$ such that Ineq.~\eqref{eq:n/4} holds.
\label{tab:pairs}}
\end{table}

\section{On the order of $k$-normal elements}
In the previous section, we proved an existence result for primitive $k$-normal elements and we found out that our technique does not work for generic values of $k$ in $[0, n]$: in fact, we can only asymptotically reach the interval $[0, n/2)$. In this section, we discuss a less restrictive question: the existence of $k$-normal elements with high multiplicative order in the group $\F_{q^n}^*$. By high order elements we mean an element $\alpha\in \F_{q^n}^*$ such that $\ord(\alpha)=e_n$, the multiplicative order of $\alpha$ in $\F_{q^n}^*$, grows faster than any polynomial in $n\log q$ when $n$ goes to infinity (see \cite{G99}). High order elements appear in the literature as an approach to primitive elements. They have been useful in many practical situations, including the celebrated AKS primality test~\cite{AKS}. Most of the well known results on high order elements give exponential expressions in $n$ as lower bounds: for instance, in the Artin-Schreier extensions $K:=\F_{p}[x]/(x^p-x-1)$, the coset of $x$ has multiplicative order at least $2^{2.54 p}$ (see \cite{V04}).

The main idea here is the following: we find estimates for the number $a_m$ of elements in $\F_{q^n}^*$ with multiplicative order at most $m$ and for the number $N_k$ of $k$-normal elements in $\fqn$. If $N_k\le a_m$, we can guarantee the existence of $k$-normal elements with multiplicative order at least $m$. We summarize our result as follows.

\begin{theorem}\label{main2}
Suppose that there exist $k$-normal elements in $\fqn$, where $n\ge 2$ and $k\le n-1$. Then there exists a $k$-normal element with multiplicative order at least 
$$\tau(k, q, n) :=4^{(k-n)/q}\cdot q^{n-k-\frac{1.1n}{\log \log (q^n-1)}}.$$
\end{theorem}
\begin{proof}
We observe that, if $G$ is a finite cyclic group, for any divisor $d$ of $|G|$ there exist $\varphi(d)$ elements in $G$ with order $d$. If $G=\F_{q^n}^*$ and $A_m$ denotes the set of elements in $\F_{q^n}^*$ with multiplicative order at most $m$, since $q^n-1\ge 3$, from Lemma~\ref{lem:estimate}, we have the following inequality:
$$a_m:=|A_m|=\sum_{d|q^n-1\atop{d\le m}}\varphi(d)< m\cdot q^{\frac{1.1n}{\log \log (q^n-1)}}.$$
If $N_k$ denotes the set of $k$-normal elements of $\F_{q^n}$ over $\F_q$, by our assumption, $N_k:=|B_k|\ne 0$. Hence, from Lemma~\ref{eq:count}, there exist some monic polynomial $h\in \F_q[x]$ of degree $n-k$ that divides $x^n-1$ and $N_k\ge \Phi_q(h)$. If $h=\prod_{i=1}^sf_i$ is the factorization of $h$ over $\F_q$ with $d_i=\deg f_i$, then $\sum_{i=1}^s d_i=n-k$ and $$\Phi_q(h)=\prod_{i=1}^s(q^{d_i}-1)\ge \prod_{i=1}^s(q-1)^{d_i}=(q-1)^{n-k}.$$ In addition, from Proposition~\ref{prop:estimate-euler}, $(q-1)^{n-k}\ge 4^{(k-n)/q}q^{n-k}$. Hence $N_k\ge 4^{(k-n)/q}q^{n-k}$. A simple calculation shows that $$m=\tau(k, q, n)=4^{(k-n)/q}\cdot q^{n-k-\frac{1.1n}{\log \log (q^n-1)}},$$ satisfies $a_m<N_k$ (this is the optimal value of $m$ based on our estimations). Therefore, there exists some element in $B_k\cap (A_m)^{c}$ and we conclude the proof.\end{proof}

\begin{remark}
We observe that $\tau(k, q, n)\ge q^{n(1-\varepsilon(q, n))-k}$, where $$\varepsilon(q, n)=\frac{1.1}{\log\log(q^n-1)}+\frac{\log 4}{q\log q}.$$
Also, $\lim\limits_{q\to\infty}\varepsilon(q, n)=0$ uniformly on $n$. In particular, for any $\varepsilon>0$, there exists $M_{\varepsilon}>0$ such that for $q>M_{\varepsilon}$, $\tau(k, q, n)>q^{n(1-\varepsilon)-k}$ for every $n\ge 2$ and $1\le k\le n-1$ .
\end{remark}

As we have seen, most of the well known results on high order elements $\alpha\in \F_{q^n}$ give lower bounds like $\text{ord}$$(\alpha)\ge c^n$, where $c>1$ is a constant that does not depend on $q$ (for more examples, see~\cite{C07} and~\cite{BR16}). In general, these high-order elements are precisely exhibited, while the previous theorem concerns only on existence. On the other hand, under the condition that $k$-normal elements actually exist, Theorem~\ref{main2} shows that there is some of them having multiplicative order fairly larger than $c^n$ when $k$ is not close to $n$. For instance, given $0<\gamma<1/2$, for $q$ sufficiently large, we have $\tau(k, q, n)> q^{\gamma n}$ for all $n$ and $k\in  [1, (1-2\gamma)n]$.

\section{Existence and number of $k$-normals}
So far we discussed the existence of $k$-normal elements with additional properties like being primitive or having large multiplicative order. Our main results are under the natural condition that $k$-normal elements actually exist in $\fqn$. We recall that, from Remark~\ref{remark:k-normals-div}, for $0\le k\le n$, there exist $k$-normal elements in $\fqn$ if and only if $x^n-1$ is divisible by a polynomial $f\in \F_q[x]$ of degree $n-k$: if this occurs, $h=(x^n-1)/f\in \F_q[x]$ divides $x^n-1$ and has degree $k$ and so we also have $(n-k)$-normal elements. If we consider $h=x^n-1$ or $h=x-1$ we see that, for $k=0, 1, n-1$ or $n$, we always have $k$-normal elements in every extension of $\F_q$. These are the only values of $k$ for which we can always ensure the existence of $k$-normals. In fact, if $n$ is a prime number and $q$ is a primitive root modulo $n$, the polynomial $E_n(x)=\frac{x^n-1}{x-1}$ is the $n$-th cyclotomic polynomial and, according to Theorem 2.47 of \cite{LN}, $E_n$ is irreducible. Hence, $x^n-1$ factors as $x^n-1=(x-1)E_n$ and therefore we do not have $k$-normal elements in $\fqn$ for $1<k<n-1$. This suggests why Conjecture~\ref{conj:mullen} only considers $0$-normal and $1$-normal elements.

Motivated by these observations, we introduce the following definition.
\begin{define}[$\F_q$-practical numbers]
A positive integer $n$ is $\F_q$-practical if, for any $1\le k\le n-1$, $x^n-1$ is divisible by a polynomial of degree $k$ over $\F_q$.
\end{define}
We observe that $n$ is $\F_q$-practical if and only if there exist $k$-normals in $\F_{q^n}$ for any $k\in [0, n]$.  This definition arises from the so called $\varphi$-{\it practical numbers}: they are the positive integers $n$ for which $x^n-1\in \mathbb Z[x]$ is divisible by a polynomial of degree $k$ for any $1\le k\le n-1$. These $\varphi$-practical numbers have been extensively studied in many aspects, such as their density over $\mathbb N$ and their asymptotic growth. In particular, if $s(t)$ denotes the number of $\varphi$-practical numbers up to $t$, according to \cite{PTW}, there exists a constant $C>0$ such that $\lim\limits_{t\to\infty}\frac{s(t)\cdot \log t}{t}=C$. This shows that the $\varphi$-practical numbers behave like the primes on integers and, in particular, their density in $\mathbb N$ is zero.

Of course, the factorization of $x^n-1$ over $\mathbb Z$ also holds over any finite field: we take the coefficients modulo $p$ and recall that $\F_p\subseteq \F_q$. This shows that any $\varphi$-practical number is also $\F_q$-practical. In particular, the number of $\F_q$-practicals up to $t$ has growth at least $\frac{Ct}{\log t}$. The exact growth of the number of $\F_q$-practicals is still an open problem. For the case $q=p$, under the Generalized Riemann Hypothesis, it was proved that the number of $\F_p$-practicals up to $t$ is about $O\left(t\sqrt{\frac{\log\log t}{\log t}}\right)$ and so, under this condition, their density is zero (see~\cite{T13}). However, there is no (unconditional) result on the density of $\F_p$-practicals over $\mathbb N$.

Nevertheless, we have that $\F_q$-practical numbers are at least as ``frequent'' as prime numbers. Considering these numbers, our main results can be applied as follows (see Theorems~\ref{main} and~\ref{main2}, Proposition~\ref{prop:suf-condition} and Corollary~\ref{cor:430}).

\begin{theorem}
Let $q$ be a prime power and let $n\ge 2$ be a positive integer that is $\F_q$-practical. For an integer $0\le k\le n$, the following hold:
\begin{enumerate}
\item If $k\le n\cdot \left(\frac{1}{2}-\frac{0.96}{\log n+\log\log q}-\log_q2 \right)$ or $q^{n/2-k}\ge W(q^n-1)W(x^n-1)$, there exist primitive $k$-normal elements in $\fqn$.
\item If the pair $(q, n)$ is in Table~\ref{tab:pairs} and $k\le n/8$, there exist primitive $k$-normal elements in $\fqn$.
\item There exists a $k$-normal element in $\fqn$ with multiplicative order at least $q^{n-k-\frac{1.1n}{\log \log (q^n-1)}}$.
\end{enumerate}
\end{theorem}

\subsection{A special class of $\F_q$-practical numbers}
We observe that the definition of $\F_q$-practical numbers is strongly related to the factorization of $x^n-1$ over $\F_q$. In general, the factorization of such polynomial is unknown: unlike in $\mathbb Z$, the polynomials $x^n-1$ may have many irreducible factors over $\F_q$. Nevertheless, under a special condition on $n$, the factorization of $x^n-1$ over $\F_q$ can be easily given.

\begin{prop}[see Corollary 1 of~\cite{BGO}]\label{prop:fabio}
Let $q$ be a prime power and let $n$ be positive integer such that every prime divisor of $n$ divides $q-1$. Additionally, suppose that $q\equiv 1\pmod 4$ if $n$ is divisible by $8$. For $m=\frac{n}{\gcd(n, q-1)}$, every irreducible factor of $x^n-1$ over $\F_q$ has degree a divisor of $m$. Additionally, for each divisor $t$ of $m$, the number of irreducible factors of $x^n-1$ having degree $t$ equals $\frac{\varphi(t)}{t}\cdot \gcd (n, q-1)=\frac{\varphi(t)n}{tm}$.
\end{prop}

In~\cite{BGO}, not only the distribution degree of the irreducible factors of $x^n-1\in \F_q[x]$ is given but we also have a complete description of the irreducible factors. In the same paper, the authors also extend the previous result, removing the restriction $q\equiv 1\pmod 4$ if $n$ is divisible by $8$. A description on the degree distribution of the irreducible factors of $x^n-1\in \F_q[x]$ can be deduced from an old result~\cite{B55}, but is not as explicit as the ones given in~\cite{BGO}. The following lemma shows the applicability of Proposition~\ref{prop:fabio} to our study on $\F_q$-practical numbers.

\begin{lemma}\label{lem:aux-practical}
Suppose that $r$ is a prime that divides $q-1$. Then for any positive integer $d$ and any $1\le k\le r^d-1$ there exists a polynomial $f\in \F_q[x]$ of degree $k$ that divides $x^{r^d}-1$ and satisfies $f(1)\ne 0$. 
\end{lemma}
\begin{proof} We observe that $r^d$ is not divisible by the characteristic of $\F_q$. Therefore, $x^{r^d}-1$ has no repeated irreducible factors.
We split the proof into cases:
\begin{enumerate}[(i)]
\item $r=2$: in this case, $x^{2^d}-1=(x-1)\prod_{i=0}^{d-1}(x^{2^i}+1)$ (not necessarily the factorization into irreducible factors over $\F_q$). If $1\le k\le 2^d-1$, it follows that $k=\sum_{i=0}^{d-1}s_i2^i$ where $s_i\in\{0, 1\}$. Therefore, $f(x)=\prod_{i=0}^{d-1}(x^{2^i}+1)^{s_i}$ has degree $k$ and divides $x^{2^d}-1$. Also, in this case, $q$ is odd and so $f(1)=2^d\ne 0$.
\item $r>2$ and $r^d$ divides $q-1$: in this case, $x^{r^d}-1=(x-1)\prod_{i=1}^{r^d-1}(x-\zeta_i)$ is the factorization of $x^{r^d}-1$ into irreducible factors over $\F_q$, where $\zeta_i\ne 1$ for any $1\le i\le r^d-1$. If $1\le k\le r^d-1$ then $f(x)=\prod_{i=1}^k(x-\zeta_i)$ has degree $k$, divides $x^{r^d}-1$ and satisfies $f(1)\ne 0$.
\item $r>2$ and $r^d$ does not divide $q-1$: in this case, $8$ does not divide $r^d$ and we employ Proposition~\ref{prop:fabio}; if $r^s$ is the greatest power of $r$ that divides $q-1$, then $1\le s<d$ and the factorization of $x^{r^d}-1$ over $\F_q$ has $r^s$ distinct irreducible factors of degree one and, for each $1\le i\le d-s$, it has $(r-1)r^{s-1}$ distinct irreducible factors of degree $r^i$.  Let $1\le k\le r^d-1$ and $(a_{d-1}\cdots a_{1}a_0)_r$ be the $r$-adic representation of $k$ with $d$ digits. For each $s+1\le i\le d$, take $a_{i-1}r^{s-1}$ distinct irreducible factors of $x^{r^d}-1$ having degree $r^{i-s}$ (this is possible since $0\le a_i\le r-1$ and $1\le i-s\le d-s$) and then take $D=\sum_{i=0}^{s-1}a_ir^i$ distinct linear factors of $x^n-1$, all different from $x-1$ (this is possible since $D\le r^s-1$). The product of the chosen factors yields a polynomial $f\in \F_q[x]$ of degree $\sum_{i=0}^{d-1}a_ir^{i}=k$ such that $f$ divides $x^{r^d}-1$ and satisfies $f(1)\ne 0$.
\end{enumerate}\end{proof}

From the previous lemma, we obtain an infinite class of $\F_q$-practical numbers.
\begin{theorem}\label{main3} Let $q$ be a power of a prime $p$ and let $n\ge 2$ be a positive integer such that every prime divisor of $n$ divides $p(q-1)$. Then $n$ is $\F_q$-practical.\end{theorem}
\begin{proof} From definition it suffices to prove that, for any $1\le k\le n-1$, there exists a polynomial $F\in \F_q[x]$ of degree $k$ that divides $x^n-1$. We prove this last statement by induction on the number $N\ge 1$ of distinct prime factors of $n$. If $N=1$ and $n=p^t$ for some $t\ge 1$ the result is trivial, otherwise $n$ is a power of a prime $r$, where $r$ divides $q-1$ and then we use Lemma~\ref{lem:aux-practical}. Suppose that the statement is true for all positive integers with at most $N$ distinct prime factors and let $n$ be a positive integer with $N+1\ge 2$ distinct prime factors. We have two cases to consider.

\begin{enumerate}[(i)]
\item $p$ divides $n$: we can write $n=p^u n_0$, where $\gcd(p, n_0)=1$, $n_0$ has $N$ distinct prime factors and every prime divisor of $n_0$ divides $q-1$. Let $k$ be a positive integer between $1$ and $n-1$. Then $k=an_0+b$, where $0\le a<p^u$ and $0\le b<n_0$. By the induction hypothesis, there exists a polynomial $g(x)\in \F_q[x]$ of degree $b$ dividing $x^{n_0}-1$ (if $b=0$, take $g(x)=1$). Therefore, $F(x)=(x^{n_0}-1)^{a}g(x)$ has degree $k$ and divides $(x^{n_0}-1)^{a+1}$, hence divides $(x^{n_0}-1)^{p^u}=x^n-1$.

\item $p$ does not divide $n$: this case is slightly different from the previous one. We can write $n=r^un_1$ where $r$ is a prime divisor of $q-1$, $\gcd(r, n_1)=1$, $n_1$ has $N$ distinct prime factors, each of them dividing $q-1$. Let $k$ be a positive integer between $1$ and $n-1$ and write $k=an_1+b$, where $0\le a<r^u$ and $0\le b<n_1$. From Lemma~\ref{lem:aux-practical}, there exists a polynomial $g(x)$ of degree $a$ that divides $x^{r^u}-1$ and satisfies $g(1)\ne 0$ (if $a=0$, take $g(x)=1$). By the induction hypothesis, there exists a polynomial $h(x)$ of degree $b$ that divides $x^{n_1}-1$ (if $b=0$, take $h(x)=1$). Then $f(x)=g(x^{n_1})h(x)$ has degree $k$ and both $h(x)$ and $g(x^{n_1})$ divide $x^{n}-1$. To finish the proof, we show that $g(x^{n_1})$ and $h(x)$ are relatively prime: if there is some element $\alpha$ in the algebraic closure of $\F_{q}$ such that $g(\alpha^{n_1})=h(\alpha)=0$, since $h(x)$ divides $x^{n_1}-1$, we have that $\alpha^{n_1}=1$ and then $g(1)=0$, a contradiction with the assumption $g(1)\ne 0$.
\end{enumerate}
\end{proof}

\subsection{The enumerator polynomial}
Fix $\F_q$ a finite field of characteristic $p$ and let $n$ be a positive integer such that $n=p^t\cdot u$ where $t\ge 0$ and $\gcd(p, u)=1$. In particular, $x^u-1$ has distinct irreducible factors over $\F_q$. If $x^u-1$ factors as $x^u-1=f_1\cdots f_s$, where 
$d_i=\deg(f_i)$, we associate to $n$ the following polynomial:
$$P_n(z):=\prod_{i=1}^{s}(1+\Phi_q(f_i)z^{d_i}+\cdots+\Phi_q(f_i^{p^t})z^{p^td_i})=\sum_{i=0}^{n}A_n(i)z^{i}\in \mathbb Z[x],$$
where $\Phi_q(f_i^{s})=(q^{d_i}-1)q^{(s-1)d_i}$ for $s\ge 1$.
Since the function $\Phi_q$ is multiplicative, we have that $$A_n(i)=\sum_{(e_1, \ldots, e_s)\in C_i}\Phi_q(f_1^{e_1}\cdots f_s^{e_s}),$$
where $C_i=\{(e_1, \ldots, e_s)\in \mathbb N^{s}\,|\, e_1d_1+\cdots+e_sd_s=i\}$. In addition, as $(e_1, \ldots, e_s)$ varies in the set $C_i$, the polynomial $f_1^{e_1}\cdots f_s^{e_s}$ runs over the divisors $h\in \F_q[x]$ of $x^n-1$ such that $\deg(h)=i$. 
Therefore, from Lemma~\ref{eq:count}, the number of $k$-normal elements in $\F_{q^n}$ equals $A_n(n-k)$. It is then interesting to study the explicit expansion of $P_n(z)$. For instance, $n$ is $\F_q$-practical if and only if all the coefficients of $P_n(z)$ are nonzero. Of course, the constant term and the leading coefficient of $P_n(z)$ can be easily computed: $A_n(0)=1$ and $A_n(n)=\Phi_q(x^n-1)$, yielding the number of $n$-normal elements and $0$-normal elements, respectively. We exemplify a very particular case when $P_n(z)$ can be explicitly computed.
\begin{example}
For $n=p^t$, $P_{n}(z)=1+(q-1)z+q(q-1)z^2+\cdots+q^{p^t-1}(q-1)z^{n}$. In particular, the number of $k$-normal elements equals $1$ if $k=n$ and $(q-1)q^{n-k-1}$ if $k\le n-1$.
\end{example}

Under the hypothesis of Proposition~\ref{prop:fabio}, the polynomial $P_n(z)$ can be implicitly computed.
\begin{cor}
Let $n$ be positive integer such that every prime divisor of $n$ divides $q-1$. Additionally, suppose that $q\equiv 1\pmod 4$ if $n$ is divisible by $8$. For $m=\frac{n}{\gcd(n, q-1)}$, we have that $$P_n(z)=\prod_{t|m}(1+(q^{t}-1)z^{t})^{\frac{\varphi(t)n}{tm}}.$$ In particular, if $n$ divides $q-1$, then $$P_n(z)=(1+(q-1)z)^n=\sum_{k=0}^{n}\binom{n}{k}(q-1)^{n-k}z^{n-k},$$ and so the number of $k$-normal elements in $\F_{q^n}$ equals $$A_n(n-k)=\binom{n}{k}(q-1)^{n-k}.$$
\end{cor}

\section{Conclusions}
In this paper, we have discussed general existence questions on the so called $k$-normal elements over finite fields, mainly motivated by the problems proposed in \cite{HMPT}. We have provided many results on the existence of $k$-normal elements with additional properties like being primitive or having large multiplicative order. In particular, we have obtained sufficient conditions for the existence of primitive $k$-normal elements in $\fqn$ whenever $k$-normal elements exist in $\fqn$; this sufficient condition is encoded in an inequality in $n, q$ and $k$ and we have provided some situations where this inequality holds, including the triples $(q, n, k)$ such that $k$ is at most $n/8$ and the pair $(q, n)$ is in Table~\ref{tab:pairs}. We also have noticed that the number of $k$-normal elements in $\fqn$ is strongly related to the factorization of $x^n-1\in \F_q[x]$ and, in general, this number can be zero. Finally, we have provided sufficient conditions on $q$ and $n$ in order to guarantee the existence of $k$-normal elements in $\fqn$ for any $0\le k\le n$ and, in some particular cases, we have also provided the number of such elements.

We emphasize that the estimates presented in this paper are turning points in the proofs of Proposition~\ref{prop:suf-condition} and Theorem~\ref{main2} but they may not be sharp: we take these estimates just because they give us satisfactory results and make the proofs more clean and simple. The upper bounds given Subsection~\ref{subsec:estimates} are true for all positive integers $m\ge 3$ but in the present text we apply these bounds for $m=q^n-1$, where $q$ is a prime power. A more detailed study on the prime factorization of $m=q^n-1$, combined with some tools in Analytic Number Theory, could yield slightly improvements on our results.


\end{document}